\documentclass[12pt,leqno]{article}
\usepackage{amsmath,amssymb,amsbsy,amsfonts,amsthm,latexsym,
         amsopn,amstext,amsxtra,euscript,amscd,amsthm,mathdots}

\allowdisplaybreaks

\numberwithin{equation}{section}

\newtheorem{lem}{Lemma}

\newtheorem{thm}{Theorem}

\newtheorem{cor}{Corollary}

\newtheorem{rem}{Remark}

\begin{document}

\begin{large}
\centerline{\Large \bf Relations in the Set of Divisors of an Integer $n$}

\end{large}
\vskip 10pt
\begin{large}
\centerline{\sc  Patrick Letendre}
\end{large}
\vskip 10pt
\begin{abstract}
Let $\mathcal{D}_{n} \subset \mathbb{N}$ be the set of the $\tau(n)$ divisors of $n$. We generalize a method developed by Erd\H os, Tenenbaum and de la Bretèche for the study of the set $\mathcal{D}_{n}$. In particular, using these ideas, we establish that
$$
|\{(d_{1},d_{2},d_{3}) \in \mathcal{D}_{n}^3 : d_{1}+d_{2}=d_{3}\}| \le \tau(n)^{2-\delta}
$$
with $\delta=0.045072$.
\end{abstract}


\section{Introduction and notation}

Let $1=t_{1}<t_{2}<\cdots<t_{\tau(n)} = n$ denote the increasing sequence of divisors of a generic integer $n$. We denote the set of these divisors by $\mathcal{D}_{n}$. In \cite{gt}, a new class of arithmetic functions was introduced to study $\mathcal{D}_{n}$. In particular, estimates were obtained for the quantities $|\{d \in \mathcal{D}_{n}:\ d(d+1)\mid n\}|$ and $|\{1 \le i \le \tau(n):\ (t_{i},t_{i+1})=1\}|$. In \cite{rdlb}, these results were significantly refined into a satisfactory theory for this class of functions. Some of the functions concerned already had a history in the literature, notably see \cite{pe:rrh} and \cite{pe:gt}. Additionally, in \cite{rdlb:2} and \cite{kg}, complementary results, specialized on divisors that are values taken by a fixed polynomial, were obtained using different methods. In this article, we propose to generalize the class of considered arithmetic functions, notably by considering $j$-tuples of coprime divisors.

We introduce the functions
$$
\kappa_{j}(n)  :=  \sum_{\substack{d_{i} \mid n \\ i=1,\dots,j \\ \gcd(d_{i_1},d_{i_2})=1 \\ (i_1 \neq i_2)}}1=\prod_{p^v \| n}(jv+1).
$$
Thus, $\kappa_{1}(n)$ is simply the function $\tau(n)$. We say that a set $U \subseteq \mathcal{D}_{n}^{j}$ is {\it regular} if for every $j$-tuple $(d_{1},\dots,d_{j}) \in U$, we have $\gcd(d_{i_1},d_{i_2})=1$ for each $1 \le i_{1} < i_{2} \le j$. In particular, we always have $|U| \le \kappa_{j}(n)$. A mapping $g_{j,n}:U_{g} \rightarrow \mathcal{D}_{n}$ defined on a certain set $U_{g} \subseteq \mathcal{D}_{n}^{j}$ that is regular is said to be $k$-{\it regular} if for every $(d_{1},\dots,d_{j}) \in U_{g}$, we have $\gcd(g_{j,n}(d_{1},\dots,d_{j}),d_{i})=1$ for each $i=1,\dots,j$, and if the following two conditions are satisfied.
\begin{enumerate}
\item[1.] For each $1 \le i \le j$ and for each choice $d_{1},\dots,d_{i-1},d_{i+1},\dots,d_{j}$ and $d$ fixed, the number of solutions $z$ to the equation
$$
g_{j,n}(d_{1},\dots,d_{i-1},z,d_{i+1},\dots,d_{j})=d
$$
with $(d_{1},\dots,d_{i-1},z,d_{i+1},\dots,d_{j}) \in U_{g}$ is at most $k$.
\item[2.] For each $1 \le i \le j$ and for each choice $d_{1},\dots,d_{i-1},d_{i+1},\dots,d_{j}$ and $d$ fixed, the number of solutions $z$ to the equation
$$
zg_{j,n}(d_{1},\dots,d_{i-1},z,d_{i+1},\dots,d_{j})=d
$$
with $(d_{1},\dots,d_{i-1},z,d_{i+1},\dots,d_{j}) \in U_{g}$ is at most $k$.
\end{enumerate}
We simply consider that the application $g_{j,n}$ on each $j$-tuple of $\mathcal{D}_{n}^{j}\setminus U_{g}$ is undefined. We denote by $\mathcal{E}_{j,k}$ the class of arithmetic functions $F_{j,k}(n)$ of the type
$$
F_{j,k}(n)=|U_{g}|
$$
for an application $g_{j,n}$ $k$-regular. Also, following \cite{rdlb}, we define
$$
E_{j,k}(n):=\max_{F_{j,k}\in\mathcal{E}_{j,k}}|F_{j,k}(n)|.
$$
In certain circumstances, we will also want to add the following third condition
\begin{enumerate}
\item[3.] For each $1 \le i_{1} < i_{2} \le j$ and for each choice of $d_{1},\dots,d_{i_{1}-1},d_{i_{1}+1},\dots,d_{i_{2}-1},$ $d_{i_{2}+1},\dots,d_{j},d$ and $d'$ fixed, the number of solutions $(z_{1},z_{2})$ to the system
$$
\left\{\begin{array}{cc}
g_{j,n}(d_{1},\dots,d_{i_{1}-1},z_{1},d_{i_{1}+1},\dots,d_{i_{2}-1},z_{2},d_{i_{2}+1},\dots,d_{j})=d\\
z_{1}z_{2}=d'\\
\end{array}\right.
$$
with $(d_{1},\dots,d_{i_{1}-1},z_{1},d_{i_{1}+1},\dots,d_{i_{2}-1},z_{2},d_{i_{2}+1},\dots,d_{j}) \in U_{g}$ is at most $k$,
\end{enumerate}
in which case we say that $g_{j,n}$ is {\it strongly} $k$-regular. Condition 3 requires $j \ge 2$ to be non-trivial. To emphasize the set of strongly $k$-regular functions, we will simply write $\mathcal{E}^{*}_{j,k}$ and $E^{*}_{j,k}(n)$ for the analogues of $\mathcal{E}_{j,k}$ and $E_{j,k}(n)$ respectively.

We will frequently use the arithmetic functions
$$
\omega(n) :=\sum_{p \mid n}1,\quad V(n) := \max_{p^{v} \| n} v,\quad \Omega(n) :=\sum_{p^{v} \| n}v \quad \mbox{and} \quad \Omega_{2}(n) :=\sum_{p^{v} \| n}v^{2}.
$$
Note that from the Cauchy-Schwarz inequality, we have $\sqrt{\Omega_{2}(n)} \ge \frac{\Omega(n)}{\sqrt{\omega(n)}}$.

\section{Main results}

We begin with the main theorem regarding $k$-regular functions.

\begin{thm}\label{thm:1}
\begin{enumerate}
\item[{\bf (a)}] Let $j \ge 1$ be a fixed integer. We introduce the constants
\begin{eqnarray*}
\delta_{j} & := & \Bigl(\frac{2j^{2}}{(j+1)(2j+1)}\log\frac{2j}{2j+1}+\frac{3j+1}{(j+1)(2j+1)}\log\frac{3j+1}{2j+1}\Bigr)/\log(j+1)\\
& = & f_{\frac{1}{2j+1}}(j)/\log(j+1) \quad (\mbox{see Lemma}\ \ref{lem:4}).
\end{eqnarray*}
For each integer $n \ge 1$,
$$ 
E_{j,k}(n) \le k\kappa_{j}(n)^{1-\delta_{j}}.
$$
\item[{\bf (b)}] Let $n \ge 1$ be an integer and $\delta := 0.045072$. Then
$$
E_{2,k}(n) \le k\kappa_{2}(n)^{1-\delta}.
$$
\end{enumerate}
\end{thm}

In some cases, Theorem \ref{thm:2} is sharper than Theorem \ref{thm:1}.

\begin{thm}\label{thm:2}
\begin{enumerate}
\item[{\bf (a)}] Let $n \ge 1$ be a squarefree integer. Then
\begin{equation}\label{eq:thm:2}
E^{*}_{j,k}(n) \le k\Bigl(\frac{j+2}{2^{\frac{2}{j+2}}}\Bigr)^{\omega(n)}.
\end{equation}

\item[{\bf (b)}] For each integer $n \ge 1$,
$$
E_{j,k}(n) \le k\prod_{p^{v}\| n}\bigl((j+1)v^{\frac{j}{j+1}}\bigr).
$$
\end{enumerate}
\end{thm}

The next result concerns the {\it additive energy} of the set $\mathcal{D}_{n}$ defined by
\begin{equation}
\mathsf{E}(\mathcal{D}_{n}):=|\{(d_{1},d_{2},d_{3},d_{4}) \in \mathcal{D}_{n}^{4}:\ d_{1}+d_{2}=d_{3}+d_{4}\}|.
\end{equation}
On one hand, the multiplicative energy of $\mathcal{D}_{n}$, as defined in \cite{pl}, is well understood; see Theorem $1.4$ of \cite{pl}. This is not surprising, considering the multiplicative structure of the elements of $\mathcal{D}_{n}$. On the other hand, the additive energy is much harder to approach. A well-known heuristic argument (see \cite{pe:cls:rt} and \cite{jcl:ks}) suggests that, when adapted to our situation, one might expect
$$
\mathsf{E}(\mathcal{D}_{n}) = 2\tau(n)^{2}+O(\tau(n)\exp(\omega(n)^{3/4+\epsilon}))
$$
for each $\epsilon > 0$, as $\omega(n)$ goes to infinity. The best-known results yield an error term of $O(\tau(n)C^{\omega(n)})$ for some large constant $C$. Theorem 3 of \cite{jhe:2} establishes the constant $C=(2^{35}\cdot 9)^{27}$.

\begin{thm}\label{thm:3}
\begin{enumerate}
\item[{\bf (a)}] Let $n \ge 1$ be a squarefree integer. We have 
$$
\mathsf{E}(\mathcal{D}_{n}) \ll (7.8784716)^{\omega(n)}.
$$
\item[{\bf (b)}] For each integer $n \ge 2$, 
$$
\mathsf{E}(\mathcal{D}_{n}) \ll \frac{\tau(n)^{3}}{\sqrt{\Omega_{2}(n)}}.
$$
\end{enumerate}
\end{thm}

For each character $\chi$ modulo $q$, we consider the function
$$
\tau(n,\chi) := \sum_{d \mid n}\chi(d).
$$
For any integer $n$ that is coprime with $q$, we define the quantity
\begin{eqnarray*}
\mathcal{H}(n,q) & := & \frac{1}{\phi(q)}\sum_{\chi \hskip-8pt\pmod{q}}|\tau(n,\chi)|^{2}\\
& = & |\{(d_{1},d_{2}) \in \mathcal{D}_{n}^{2}:\ d_{1}\equiv d_{2}\ \pmod{q}\}|,
\end{eqnarray*}
where $\phi(q)$ denotes the Euler totient function. The next result establishes that most of the functions $\tau(n,\chi)$ are much smaller than $\tau(n)$ provided only that $q$ is large enough. Some restriction has to be expected since the prime divisors of $n$ could predominantly reside in a small multiplicative subgroup of $\mathbb{Z}/q\mathbb{Z}$.

\begin{thm}\label{thm:4}
Let $n,q \ge 2$ be fixed coprime integers. Then
$$
\mathcal{H}(n,q) \ll \bigl(\tau(n)+\tau(n)^{2-4\eta}\bigr)V(n)(\log \tau(n))^{3/2}
$$
where $\eta:=\frac{\log q}{\log n}$.
\end{thm}

The next result illustrates an application of Theorem \ref{thm:1}. It can be compared with the estimate $3\cdot 7^{3}\cdot 49^{\omega(n)}\cdot \tau(n)$ which is derived from Theorem 1 in the classic paper of Evertse \cite{jhe:1}.

\begin{cor}\label{cor:1}
Let $n \ge 1$ be an integer and $\delta$ be the constant introduced in Theorem \ref{thm:1}b. Then
\begin{equation}\label{eq:1}
|\{(d_{1},d_{2},d_{3}) \in \mathcal{D}_{n}^{3}:\ d_{1}+d_{2}=d_{3}\}| \le \tau(n)^{2-\delta}.
\end{equation}
\end{cor}

Corollary \ref{cor:2} generalizes Corollaire 2.1 of \cite{rdlb}.

\begin{cor}\label{cor:2}
Let $j \ge 1$ be a fixed integer. For each integer $n \ge 2$,
$$
E_{j,k}(n) \ll \frac{k\kappa_{j}(n)}{\Omega(n)}.
$$
\end{cor}

The last corollary can be viewed as a generalization of Corollary \ref{cor:1} in the case where $n$ is squarefree. It follows easily from Theorem \ref{thm:3}.

\begin{cor}\label{cor:3}
Let $n \ge 1$ be a squarefree integer. The function
$$
G(m):=|\{(d_{1},d_{2},d_{3}) \in \mathcal{D}_{n}^3 : d_{1}+d_{2}=d_{3}+m\}|
$$
satisfies
$$
G(m) \ll (3.969502)^{\omega(n)}
$$
uniformly in $m$.
\end{cor}

\section{Preliminary lemmas}

Let us write the factorization of the integer $n$ into distinct prime numbers as
$$
n=p_{1}^{v_{1}}\cdots p_{\omega(n)}^{v_{\omega(n)}}
$$
where $\nu(n):=(v_{1},\dots,v_{\omega(n)})$ satisfies $v_{1} \ge \cdots \ge v_{\omega(n)}$. The vector $\nu(n)$ is uniquely defined, even if there could be many such factorization of $n$.

\begin{lem}\label{lem:1}
Let $j,k \ge 1$ be fixed integers and $n \ge 1$ be a generic integer. Then, $E_{j,k}(n)$ depends only on $\nu(n)$.
\end{lem}

\begin{proof}
Let $m$ be another integer for which $\nu(m)=\nu(n)$ and $g_{j,n}$ be a $k$-regular function that satisfies $F_{j,k}(n)=E_{j,k}(n)$. We write 
$$
n=p_{1}^{v_{1}}\cdots p_{\omega(n)}^{v_{\omega(n)}}\quad\mbox{and}\quad m=q_{1}^{v_{1}}\cdots q_{\omega(n)}^{v_{\omega(n)}}.
$$
For each divisor $d$ of $n$ of the form $d=p_{1}^{\gamma_{1}}\cdots p_{\omega(n)}^{\gamma_{\omega(n)}}$ with $0 \le \gamma_{i} \le v_{i}$ we denote by $\tilde{d}$ the integer $q_{1}^{\gamma_{1}}\cdots q_{\omega(n)}^{\gamma_{\omega(n)}}$. Thus, we can define the function
$$
g_{j,m}(\tilde{d}_{1},\dots,\tilde{d}_{j}):= \tilde{g}_{j,n}(d_{1},\dots,d_{j}).
$$
We verify that $g_{j,m}$ satisfies the desired conditions and we find $E_{j,k}(n) \le E_{j,k}(m)$. The result follows.
\end{proof}

\begin{lem}\label{lem:2}
Let $m,n \ge 1$ be coprime integers. Then
$$
E_{j,k_{1}}(n)E_{j,k_{2}}(m) \le E_{j,k_{1}k_{2}}(nm).
$$
\end{lem}

\begin{proof}
Let $g_{j,n}$ and $g_{j,m}$ be functions that realize $F_{j,k_{1}}(n)=E_{j,k_{1}}(n)$ and $F_{j,k_{2}}(m)=E_{j,k_{2}}(m)$ respectively. For each $(d_{1},\dots,d_{j}) \in U_{g_{j,n}}$ and $(e_{1},\dots,e_{j}) \in U_{g_{j,m}}$ we define
$$
g_{j,nm}(d_{1}e_{1},\dots,d_{j}e_{j}):=g_{j,n}(d_{1},\dots,d_{j})g_{j,m}(e_{1},\dots,e_{j}).
$$
We verify that $g_{j,nm}$ is $k_{1}k_{2}$-regular and the result follows.
\end{proof}

\begin{lem}\label{lem:3}
Let $j \ge 1$ be a fixed integer. Assume that for each $k,n \ge 1$ we have
$$
E_{j,k}(n) \le Ck\Upsilon(n)
$$
for some constant $C \ge 1$ and some multiplicative function $\Upsilon$ that depends only on $\nu(n)$. Then
$$
E_{j,k}(n) \le k\Upsilon(n).
$$
\end{lem}

\begin{proof}
By Lemma \ref{lem:1}, we can find $N-1$ pairwise coprime integers $n_{i}$ that are also coprime with $n$, with the same vector $\nu$ such that
$$
E_{j,k}(n)=E_{j,k}(n_{1})=\dots=E_{j,k}(n_{N-1}).
$$
Applying Lemma \ref{lem:2}, we find
\begin{eqnarray*}
E^{N}_{j,k}(n) & = & E_{j,k}(n)\prod_{i=1}^{N-1}E_{j,k}(n_{i})\\
& \le & E_{j,k^{N}}(nn_{1}\cdots n_{N-1})\\
& \le & C k^{N}\Upsilon(nn_{1}\cdots n_{N-1})\\
& = & C k^{N}\Upsilon^{N}(n).
\end{eqnarray*}
Thus, $E_{j,k}(n) \le C^{1/N} k\Upsilon(n)$ and, as $N$ approaches infinity, we obtain the desired result.
\end{proof}

\begin{rem}\label{rem_*}
Lemma \ref{lem:1}, \ref{lem:2} and \ref{lem:3} still hold for $E^{*}_{j,k}(n)$.
\end{rem}

\begin{lem}\label{lem:4}
For each real $\alpha \in [0,1)$ we consider the functions
$$
f_{\alpha}(x):=-(1-\alpha)\frac{x}{x+1}\log\frac{1}{1-\alpha}+\frac{\alpha x+1}{x+1}\log(\alpha x+1)\quad(x\ge 0)
$$
and
$$
\ell_{\alpha}(x):=\frac{(1+\alpha)x}{x+1}\log\frac{\alpha x+1}{1-\alpha}-\log\frac{\alpha (x-1)+1}{1-\alpha}\quad(x \ge 1).
$$
We have
\begin{enumerate}
\item[$(i)$] $f_{\alpha}(0)=0$;

\item[$(ii)$] The function $\frac{f_{\alpha}(x)}{\log(x+1)}$ is strictly increasing for $x \ge 0$ for each $\alpha \in (0,1)$;

\item[$(iii)$] $\ell_{\alpha}(x) \ge f_{\alpha}(x)$ for each $x \ge 1$.
\end{enumerate}
\end{lem}

\begin{proof}
$(i)$ is an easy verification. For the proof of $(ii)$, we proceed as follows.
\begin{enumerate}
\item[$\bullet$] We consider the function
$$
b_{1}(\alpha,x):=(x+1)^{2}(\log(x+1))^{2}\frac{d}{dx}\frac{f_{\alpha}(x)}{\log(x+1)}.
$$
We have to show that $b_{1}(\alpha,x)$ is strictly positive for each $x > 0$. We verify that $b_{1}(\alpha,0)=0$.

\item[$\bullet$] We then consider the function
$$
b_{2}(\alpha,x):=(x+1)(\alpha x+1)\frac{d}{dx}b_{1}(\alpha,x).
$$
We have to show that $b_{2}(\alpha,x)$ is strictly positive for each $x > 0$. We verify that $b_{2}(\alpha,0)=0$.

\item[$\bullet$] We verify that the function
$$
\frac{d}{dx}b_{2}(\alpha,x)\Bigl|_{x=0} = -(1-\alpha)(\alpha+\log(1-\alpha))
$$
is strictly positive for each $\alpha \in (0,1)$ by evaluating the rightmost factor at $\alpha=0$ and noting that its derivative is positive.

\item[$\bullet$] We verify that
$$
\frac{d^{2}}{dx^{2}}b_{2}(\alpha,x)\Bigl|_{x=0} = -2\alpha(1-\alpha)\log(1-\alpha)
$$
is strictly positive for each $\alpha \in (0,1)$.

\item[$\bullet$] We verify that
$$
\frac{d^{3}}{dx^{3}}b_{2}(\alpha,x) = \frac{2\alpha^{2}(1-\alpha)}{(x+1)(\alpha x+1)}
$$
is strictly positive for each $\alpha \in (0,1)$ and $x \ge 0$.
\end{enumerate}
By going back up the chain of steps, we deduce that $\frac{f_{\alpha}(x)}{\log(x+1)}$ is strictly increasing for $x \ge 0$ for each $\alpha \in (0,1)$. For the proof of $(iii)$, we write
$$
q(\alpha,x):=\ell_{\alpha}(x)-f_{\alpha}(x) \quad (x \ge 1)
$$
and we proceed as follows.
\begin{enumerate}
\item[$\bullet$] We first verify that $q(\alpha,1) \equiv 0$.

\item[$\bullet$] We then consider the function
$$
q_{1}(\alpha,x):=(x+1)^{2}(\alpha x+1)(\alpha(x-1)+1)\frac{d}{dx}q(\alpha,x).
$$
We have to show that $q_{1}(\alpha,x)$ is positive for $x \ge 1$ for each $\alpha \in [0,1)$. We verify that
$$
q_{1}(\alpha,1) = 2(\alpha+1)(\log(\alpha+1)-\log(1-\alpha)-2\alpha)
$$
is positive for each $\alpha \in [0,1)$ by evaluating the rightmost factor at $\alpha=0$ and noting that its derivative is positive.

\item[$\bullet$] We now turn to the first derivative
$$
q'_{1}(\alpha,x):=\frac{d}{dx}q_{1}(\alpha,x).
$$
We verify that
$$
q'_{1}(\alpha,1) = 2\alpha(\alpha+2)\Bigl(\log(\alpha+1)-\log(1-\alpha)-\frac{4\alpha}{\alpha+2}\Bigr)
$$
is positive for each $\alpha \in [0,1)$ as previously.

\item[$\bullet$] We now turn to the second derivative
$$
q''_{1}(\alpha,x):=\frac{d}{dx}q'_{1}(\alpha,x).
$$
We verify that
$$
q''_{1}(\alpha,1) = 4\alpha^{2}\Bigl(\log(\alpha+1)-\log(1-\alpha)-\frac{\alpha}{2(\alpha+1)}\Bigr)
$$
is positive $\alpha \in [0,1)$.

\item[$\bullet$] We now turn to the third derivative
\begin{eqnarray*}
q'''_{1}(\alpha,x) & := & \frac{d}{dx}q''_{1}(\alpha,x)\\
& = & \frac{2\alpha^{3}(2\alpha x+\alpha+2)}{(\alpha x+1)^{2}}
\end{eqnarray*}
which is clearly positive for each $\alpha \in [0,1)$ and $x \ge 1$.
\end{enumerate}
By going back up the chain of steps, we deduce that $\ell_{\alpha}(x) \ge f_{\alpha}(x)$ for $x \ge 1$ for each $\alpha \in [0,1)$. This completes the proof.
\end{proof}

\begin{lem}\label{lem:5}
For every integer $s \ge 0$ and real $1 \le \beta \le 2$,
$$
\sum_{i=0}^{s}(2i+1)^{\beta-1} \le (s+1)^{\beta}.
$$
\end{lem}

\begin{proof}
We proceed by induction. The case $s = 0$ is straightforward. For the inductive step, we consider the function
$$
\psi(x):=\frac{(x+1)^{\beta}-x^{\beta}}{(2x+1)^{\beta-1}}-1.
$$
It is enough to establish that $\psi(s) \ge 0$ for $s \ge 1$. We have $\psi(0) = 0$ and
$$
\psi'(x) = \beta\frac{(x+1)^{\beta-1}-x^{\beta-1}}{(2x+1)^{\beta-1}}-2(\beta-1)\frac{(x+1)^{\beta}-x^{\beta}}{(2x+1)^{\beta}}
$$
so that $\psi'(0) = 2-\beta \ge 0$ and $\psi'(\infty)=0$. Differentiating a second time gives
$$
\psi''(x) =-\beta(\beta-1) \frac{x^{\beta-2}-(x+1)^{\beta-2}}{(2x+1)^{\beta+1}}
$$
which clearly satisfies $\psi''(x) \le 0$ for $1 \le \beta \le 2$ and $x > 0$. We deduce that $\psi'(x)$ is positive for all $x \ge 0$ and then that $\psi(x)$ is positive for all $x \ge 0$. This completes the proof.
\end{proof}

The Erd\H os-Hooley Delta function (see \cite{ch}) is defined by
$$
\Delta(n):=\max_{u \in \mathbb{R}}|\{d \mid n:\ e^{u} < d \le e^{u+1}\}|.
$$
\begin{lem}\label{lem:6}
For every integer $n \ge 2$,
$$
\Delta(n) \ll \frac{\tau(n)}{\sqrt{\Omega_{2}(n)}}.
$$
\end{lem}

\begin{proof}
Define the set $R_{n}(X)=\{d \mid n:\ X < d \le 2X\}$ and let $X$ be a fixed real number. Any two distinct elements $d_{1}$ and $d_{2}$ of $R_{n}(X)$ cannot divide one another, that is $d_{1} \nmid d_{2}$. It follows from Theorem 1 of \cite{ngdb:cvet:dk} that $|R_{n}(X)|$ is at most equal to the number of divisors $d$ of $n$ that satisfy $\Omega(d)=\lfloor\frac{\Omega(n)}{2}\rfloor$. From Theorem 2 of \cite{ia}, we deduce that $|R_{n}(X)| \ll \frac{\tau(n)}{\sqrt{\Omega_{2}(n)}}$. The result then follows from the inequality
$$
\sum_{\substack{d \mid n \\ e^{u} < d \le e^{u+1}}} 1 \le \sum_{\substack{d \mid n \\ e^{u} < d \le 2e^{u}}} 1+\sum_{\substack{d \mid n \\ 2e^{u} < d \le 4e^{u}}} 1.
$$
\end{proof}

\begin{rem}
It is interesting to note that the previous lemma implies that the number of solutions to the equation
$$
n = x^{k}+y^{k} \qquad ((x,y) \in \mathbb{Z}^{2}_{\ge 0})
$$
is at most $\ll \frac{k\tau(n)}{\sqrt{\Omega_{2}(n)}}$ for each odd integer $k \ge 3$ when $n \ge 2$. It is well-known to be false for $k=2$.
\end{rem}

\begin{lem}\label{lem:7}
Let $(x_{i},y_{i})$ for $i=1,\dots,k$ be pairs of positive real numbers that satisfy $\sum_{i=1}^{k}x_{i}\le\sum_{i=1}^{k}y_{i}$. There exists a permutation $\sigma=(\sigma_{1},\dots,\sigma_{k})$ such that for each $1 \le s \le k$,
$$
\sum_{i=1}^{s}x_{\sigma_{i}} \le \sum_{i=1}^{s}y_{\sigma_{i}}.
$$
\end{lem}

\begin{proof}
We construct $\sigma$ with an inductive process. Assume that the $i$-th element of $\sigma$ is determined for $i=t+1,\dots,k$ and that we have established the inequality
\begin{equation}\label{eq:7.1}
\sum_{i=1}^{t}x_{\sigma_{i}} \le \sum_{i=1}^{t}y_{\sigma_{i}}
\end{equation}
where, even if the order is unknown, the sums are well defined. If $t=1$, we define $\sigma_{t}$ to be the last index, thus concluding the inductive process. Otherwise, we assume for a contradiction that for each subset $I$ of $t-1$ of the $t$ remaining subscripts we have
$$
\sum_{i \in I}x_{i} > \sum_{i \in I}y_{i}.
$$
By summing these inequalities over all possible $\binom{t}{t-1}=t$ such sets $I$, we get a contradiction with \eqref{eq:7.1} from the fact that each subscript is in exactly $t-1$ of these sets. This contradiction implies that there exists a subset $I_{0}$ of size $t-1$ such that
$$
\sum_{i \in I_{0}}x_{i} \le \sum_{i \in I_{0}}y_{i}.
$$
Thus, we define $\sigma_{t}$ as the remaining index, the one that is not included in $I_{0}$. This process eventually terminates, and thus the proof is complete. 
\end{proof}

\begin{lem}\label{lem:8}
Let $n,q \ge 2$ be fixed coprime integers, and assume that $n^{\frac{1}{4}+\mu} \le q \le n$ for some $\mu > 0$. For each congruence class $t$ modulo $q$, we have
$$
|\{d \in \mathcal{D}_{n}:\ d \equiv t \pmod{q}\}| \ll \frac{1}{\mu^{3/2}}.
$$
\end{lem}

\begin{proof}
The result follows from Corollary 3.1 of \cite{dc:nhg:svn}, see also \cite{hwlj} and \cite{hc}.
\end{proof}

\section{Proof of the theorems}
\subsection{Proof of Theorem \ref{thm:1}}

Following the proof of Théorème 2a of \cite{rdlb}, we introduce an additive function $h_{\alpha,j,n}(p^{i}):=u_{\alpha,j}(v)$, where $p^{v} \| n$, to be specified later. The average of $h_{\alpha,j,n}(d_{s})$ $(s=1,\dots,j)$ is given by
\begin{eqnarray*}
A_{\alpha,j}(n) & := & \frac{1}{\kappa_{j}(n)}\sum_{\substack{d_i \mid n \\ i=1,\dots,j \\ \gcd(d_{i_1},d_{i_2})= 1\\ (i_1 \neq i_2)}}h_{\alpha,j,n}(d_{s})\\
& = & \sum_{p^{v} \| n}\frac{vu_{\alpha,j}(v)}{jv+1}.
\end{eqnarray*}
For fixed weights $\lambda_{s} \ge 0$ such that $\lambda_{1}+\cdots+\lambda_{j}=1$, we want to compare the linear form $\lambda_{1}h_{\alpha,j,n}(d_{1})+\cdots+\lambda_{j}h_{\alpha,j,n}(d_{j})$ to its average $A_{\alpha,j}(n)$. We will choose $h_{\alpha,j,n}$ to optimize the concentration of its values around $A_{\alpha,j}(n)$.

We introduce the function
{\footnotesize$$
S^{-}_{j}(n,\alpha)  :=  |\{(d_{1},\dots,d_{j}) \in \mathcal{D}_{n}^{j}:\ \gcd(d_{i_1},d_{i_2})= 1\ (i_1 \neq i_2),\ \frac{1}{j}h_{\alpha,j,n}(d_{1}\cdots d_{j}) \le (1-\alpha)A_{\alpha,j}(n)\}|
$$}
\hskip -4pt for $\alpha \in [0,1)$.

We can write
\begin{eqnarray*}
S^{-}_{j}(n,\alpha) & \le & \sum_{\substack{d_i \mid n \\ i=1,\dots,j \\ \gcd(d_{i_1},d_{i_2})= 1\\ (i_1 \neq i_2)}}\exp\Bigl((1-\alpha)A_{\alpha,j}(n)-\frac{1}{j}h_{\alpha,j,n}(d_{1}\cdots d_{j})\Bigr)\\
& = & \kappa_{j}(n)\exp\Bigl(\sum_{p^{v} \| n}(1-\alpha)\frac{vu_{\alpha,j}(v)}{jv+1}+\log\frac{1+jv\exp(-u_{\alpha,j}(v)/j)}{jv+1}\Bigr).
\end{eqnarray*}
We find that the optimal choice is
\begin{equation}
u_{\alpha,j}(v) := j\log\frac{\alpha jv+1}{1-\alpha},
\end{equation}
which yields the inequality
\begin{equation}\label{s-}
S^{-}_{j}(n,\alpha) \le \kappa_{j}(n)\exp\Bigl(-\sum_{p^{v} \| n}f_{\alpha}(jv)\Bigr),
\end{equation}
where
$$
f_{\alpha}(x):=-(1-\alpha)\frac{x}{x+1}\log\frac{1}{1-\alpha}+\frac{\alpha x+1}{x+1}\log(\alpha x+1).
$$

Next, we consider
\begin{eqnarray*}
S^{+}_{j}(n,\alpha,\beta,r) & := & |\{(d_{1},\dots,d_{j}) \in \mathcal{D}_{n}^{j}:\ \gcd(d_{i_1},d_{i_2})= 1\ (i_1 \neq i_2),\\
&& \quad\frac{h_{\alpha,j,n}(d_{1})}{1+(j-1)r}+\frac{rh_{\alpha,j,n}(d_{2}\cdots d_{j})}{1+(j-1)r} \ge (1+\beta)A_{\alpha,j}(n)\}|.
\end{eqnarray*}
With $h_{\alpha,j,n}$ fixed, we have
{\scriptsize\begin{eqnarray}\nonumber
S^{+}_{j}(n,\alpha,\beta,r) & \le & \sum_{\substack{d_i \mid n \\ i=1,\dots,j \\ \gcd(d_{i_1},d_{i_2})= 1\\ (i_1 \neq i_2)}}\exp\Bigl(\frac{h_{\alpha,j,n}(d_{1})}{1+(j-1)r}+\frac{rh_{\alpha,j,n}(d_{2}\cdots d_{j})}{1+(j-1)r}-(1+\beta)A_{\alpha,j}(n)\Bigr)\\ \label{eq:th1b}
& = & \kappa_j(n)\exp\Bigl(\sum_{p^v \| n}\log\frac{1+v\exp(\frac{u_{\alpha,j}(v)}{1+(j-1)r})+(j-1)v\exp(\frac{ru_{\alpha,j}(v)}{1+(j-1)r})}{jv+1}-(1+\beta)\frac{vu_{\alpha,j}(v)}{jv+1}\Bigr).
\end{eqnarray}}
\hskip -4pt This leads to the definition of the function
{\footnotesize$$
\xi(v,\alpha,\beta,j,r):=-\log\frac{1+v\exp(\frac{j\log\frac{\alpha jv+1}{1-\alpha}}{1+(j-1)r})+(j-1)v\exp(\frac{rj\log\frac{\alpha jv+1}{1-\alpha}}{1+(j-1)r})}{jv+1}+(1+\beta)\frac{jv\log\frac{\alpha jv+1}{1-\alpha}}{jv+1}
$$}
\hskip -3pt for which a lower estimate is required for $\frac{\xi(x,\alpha,\beta,j,r)}{\log(jx+1)}$ for $x \ge 1$ when $\alpha \in [0,1)$, $r \ge 0$ and $j \in \mathbb{N}$ are fixed.

\subsubsection*{Completion of the proof of Theorem \ref{thm:1}a}

Let us introduce the function $\ell_{\alpha}(x):=\xi(x/j,\alpha,\alpha,j,1)$, that is
$$
\ell_{\alpha}(x):=\frac{(1+\alpha)x}{x+1}\log\frac{\alpha x+1}{1-\alpha}-\log\frac{\alpha (x-1)+1}{1-\alpha}.
$$
By $(iii)$ of Lemma \ref{lem:4}, the function $q(\alpha,x)=\ell_{\alpha}(x)-f_{\alpha}(x)$ is positive for $x \ge 1$ and $\alpha \in [0,1)$. Therefore,
\begin{eqnarray}\nonumber
S^{+}_{j}(n,\alpha,\alpha,1) & \le & \kappa_{j}(n)\exp\Bigl(-\sum_{p^{v} \| n}\ell_{\alpha}(jv)\Bigr)\\ \label{s+}
& \le & \kappa_{j}(n)\exp\Bigl(-\sum_{p^{v} \| n}f_{\alpha}(jv)\Bigr).
\end{eqnarray}

To conclude the proof, we choose a $k$-regular function $g_{j,n}$ for which $F_{j,k}(n)=E_{j,k}(n)$. We set $\alpha=\frac{1}{2j+1}$. We observe that for each $j$-tuple counted by $F_{j,k}(n)$ we either have $h_{\alpha,j,n}(d_{1}\cdots d_{j})/j$ or one of the
$$
\frac{1}{j}h_{\alpha,j,n}\Bigl(\frac{d_{1}\cdots d_{j}g_{j,n}(d_{1},\dots,d_{j})}{d_{i}}\Bigr) \qquad (i \in \{1,\dots,j\})
$$
that is at most $(1-\alpha)A_{\alpha,j}(n)$ or else we have $h_{\alpha,j,n}(d_{1}\cdots d_{j-1}(d_{j}g_{j,n}(d_{1},\dots,d_{j})))/j$ that is at least $(1+\alpha)A_{\alpha,j}(n)$. Since $g_{j,n}$ is $k$-regular, we deduce from \eqref{s-}, \eqref{s+} and Lemma \ref{lem:4} $(ii)$ that
$$
E_{j,k}(n) \le ((j+1)k+1)\kappa_{j}(n)^{1-\delta_{j}}.
$$
The result follows from Lemma \ref{lem:3}.

\subsubsection*{Completion of the proof of Theorem \ref{thm:1}b}

For the case $j=2$ we will refine our argument. We set the parameters as follows:
$$
\alpha:=0.2288541994,\quad r:=0.692466598\quad \mbox{and}\quad \beta:=\frac{2+r}{1+r}(1-\alpha)-1.
$$
We then verify that $f_{\alpha}(2)/\log(3) \ge 0.045072 (= \delta)$ so that \eqref{s-} becomes
$$
S^{-}_{2}(n,\alpha) \le \kappa_{2}(n)^{1-\delta} 
$$
from Lemma \ref{lem:4} $(ii)$. Similarly, we will show that
\begin{equation}\label{todo}
\xi(v,\alpha,\beta,2,r)/\log(2v+1) \ge \delta
\end{equation}
for each $v \ge 1$ leading to
$$
S^{+}_{2}(n,\alpha,\beta,r) \le \kappa_{2}(n)^{1-\delta}
$$
from \eqref{eq:th1b}.

Assuming the inequality \eqref{todo}, the remainder of the argument is similar to the previous one. Precisely, we either have one of the terms
$$
\frac{h_{\alpha,2,n}(d_{1}d_{2})}{2},\ \frac{h_{\alpha,2,n}(d_{1}g_{2,n}(d_{1},d_{2}))}{2}\ \mbox{and}\ \frac{h_{\alpha,2,n}(d_{2}g_{2,n}(d_{1},d_{2}))}{2}
$$
being at most $(1-\alpha)A_{\alpha,2}(n)$ or else one of the terms
$$
\frac{h_{\alpha,2,n}(d_{1}g_{2,n}(d_{1},d_{2}))}{1+r}+\frac{rh_{\alpha,2,n}(d_{2})}{1+r}\ \mbox{or}\ \frac{h_{\alpha,2,n}(d_{2}g_{2,n}(d_{1},d_{2}))}{1+r}+\frac{rh_{\alpha,2,n}(d_{1})}{1+r}
$$
is at least $\frac{2+r}{1+r}(1-\alpha)A_{\alpha,2}(n)$. Since $g_{2,n}$ is $k$-regular we conclude
$$
E_{2,k}(n) \le (4k+1)\kappa_{2}(n)^{1-\delta}
$$
and the result follows from Lemma \ref{lem:3}.

It remains only to establish \eqref{todo}. Using a computer, we first verify that the inequality holds for each $v < 1000000$. For $v \ge 1000000$, we show that the numerator in the first logarithm is a most $0.11994463\cdot(2v+1)^{2.181707220906701759}$ (the first exponential dominate) and we establish that the term in the second logarithm is at least $0.29677163413447\cdot(2v+1)$. We then find the lower estimate
$$
0.0450722+\frac{0.63}{\log(2v+1)} \qquad (v \ge 1000000)
$$
and the result follows.

\subsection{Proof of Theorem \ref{thm:2}a}

For every choice of divisors $d_{1},\dots,d_{j}$ for which $g_{j,n}(d_{1},\dots,d_{j})$ is well defined, we implicitly define $e$ by
\begin{equation}\label{fac_n}
n =: d_{1}\cdots d_{j}g_{j,n}(d_{1},\dots,d_{j})e.
\end{equation}

We can verify that any choice of $j$ divisors from the $j+2$ terms in this last product allows to solve for the complete system $d_{1},\dots,d_{j},e$ with at most $k$ solutions. One of this choice is formed of no more than $\frac{j\omega(n)}{j+2}$ prime numbers. Since there are $j^{i}$ ways to write an integer formed of $i$ prime numbers as a product of $j$ integers, we deduce that
\begin{eqnarray*}
E^{*}_{j,k}(n) & \le & k\binom{j+2}{j}\sum_{0 \le i \le \frac{j\omega(n)}{j+2}}\binom{\omega(n)}{i}j^{i}\\
& \ll & \frac{k}{\sqrt{\omega(n)}}\Bigl(\frac{j+2}{2^{\frac{2}{j+2}}}\Bigr)^{\omega(n)}.
\end{eqnarray*}
To derive the second line from the first, we have first computed the ratio of two consecutive terms in the sum. We have then obtained the inequality using a geometric series and Stirling's formula. The result then follows from Remark \ref{rem_*}.

\begin{rem}\label{rem_imp}
We can establish the same result for $E_{j,k}(n)$ in the cases $j=1$ and 2. For $j=1$, this is simply Theorem \ref{thm:1}a, so we will assume that $j \ge 2$ in what follows. For $j \ge 3$, our proof works, but it provides a larger constant. Indeed, following the same notation, we will distinguish two cases whose contributions will be denoted $E^{(1)}_{j,k}(n)$ and $E^{(2)}_{j,k}(n)$ respectively. Let $\theta_{j}$ be chosen later such that $\frac{j}{j+2}\le\theta_{j} < \frac{j}{j+1}$. In the first case, we have either $\omega(d_{1}\cdots d_{j}) \le \theta_{j}\omega(n)$ or
$$
\omega\Bigl(\frac{d_{1}\cdots d_{j}g_{j,n}(d_{1},\dots,d_{j})}{d_{i}}\Bigr) \le \theta_{j}\omega(n)
$$
for a certain $i$ in $\{1,\dots,j\}$. In this case, since there are $j^{i}$ ways to express an integer formed from $i$ distinct primes as a product of $j$ integers, we deduce that
\begin{eqnarray*}
E^{(1)}_{j,k}(n) & \le & (jk+1)\sum_{0 \le i \le \theta_{j}\omega(n)}\binom{\omega(n)}{i}j^{i}\\
& \ll & \frac{k}{\sqrt{\omega(n)}}\Bigl(\frac{j^{\theta_{j}}}{\theta_{j}^{\theta_{j}}(1-\theta_{j})^{1-\theta_{j}}}\Bigr)^{\omega(n)}.
\end{eqnarray*}

Now, considering the second case, we have
$$
\omega(d_{i}) > \frac{\theta_{j}\omega(n)}{j}
$$
for at least one $i$ in $\{1,\dots,j\}$. Let $i_{0}$ be one of these values, we also have 
$$
\omega\Bigl(\frac{d_{1}\cdots d_{j}g_{j,n}(d_{1},\dots,d_{j})}{d_{i_{0}}}\Bigr) > \theta_{j}\omega(n)
$$
By reasoning as above, we get
\begin{eqnarray*}
E^{(2)}_{j,k}(n) & \le & jk\sum_{i \ge \theta_{j}\omega(n)}\sum_{s \ge \theta_{j}\omega(n)/j}\binom{\omega(n)}{i}\binom{\omega(n)-i}{s}(j-1)^{i}\\
& \ll & k\sqrt{\omega(n)}\sum_{i \ge \theta_{j}\omega(n)}\binom{\omega(n)}{i}\binom{\omega(n)-i}{\lceil\theta_{j}\omega(n)/j\rceil}(j-1)^{i}\\
& = & k\sqrt{\omega(n)}\binom{\omega(n)}{\lceil\theta_{j}\omega(n)/j\rceil}\sum_{i \ge \theta_{j}\omega(n)}\binom{\omega(n)-\lceil\theta_{j}\omega(n)/j\rceil}{i}(j-1)^{i}\\
& \ll & k\sqrt{\omega(n)}\binom{\omega(n)}{\lceil\theta_{j}\omega(n)/j\rceil}\binom{\omega(n)-\lceil\theta_{j}\omega(n)/j\rceil}{\lceil\theta_{j}\omega(n)\rceil}(j-1)^{\theta_{j}\omega(n)}\\
& \ll & \frac{k}{\sqrt{\omega(n)}}\left(\frac{(j-1)^{\theta_{j}}}{(\theta_{j}/j)^{\theta_{j}/j}\theta_{j}^{\theta_{j}}(1-\theta_{j}-\theta_{j}/j)^{1-\theta_{j}-\theta_{j}/j}}\right)^{\omega(n)}.
\end{eqnarray*}
We then choose $\theta_{j}$ to balance the contributions from both cases and the result follows from Lemma \ref{lem:3}.

For $j=2$, the optimal choice is $\theta_{2}=\frac{1}{2}$, which means that $E_{2,k}(n)$ also satisfies the inequality \eqref{eq:thm:2} as desired. For $j \ge 3$, we do not claim to have presented the best possible result here. However, we found no argument that could yield the inequality \eqref{eq:thm:2} for $E_{j,k}(n)$. Specifically, for $j=3$, an example of a limiting case is given by
$$
(\omega(d_{1}),\omega(d_{2}),\omega(d_{3}),\omega(g_{3,k}(d_{1},d_{2},d_{3})),\omega(e)) \approx \Bigl(\frac{1}{5}+\epsilon,\frac{1}{5}+\epsilon,\frac{1}{5}+\epsilon,\frac{1}{5}-\frac{3\epsilon}{2},\frac{1}{5}-\frac{3\epsilon}{2}\Bigr)\cdot \omega(n)
$$
where $\epsilon > 0$ is chosen to be sufficiently small. It is at this stage that we chose to introduce the notion of strongly $k$-regular functions in the hope of eventually getting a nice theory.
\end{rem}

\begin{rem}
Given the previous observation and the well-understood case of $j=1$, it is plausible that the limiting exponent for $j=2$ also occurs when $v=1$. In this scenario, we would have
$$
E_{2,k}(n) \le k\kappa_{2}(n)^{1-\delta'},
$$
where
$$
\delta':=1-\frac{3}{2}\frac{\log 2}{\log 3} = 0.0536053696\dots
$$
\end{rem}

\subsection{Proof of Theorem \ref{thm:2}b}

Given an integer $n$, we use the function
$$
T(d):=\prod_{\substack{p \mid d \\ p^{v} \| n}}\frac{1}{v} \qquad (d \in \mathcal{D}_{n})
$$
as introduced in \cite{rdlb}. We observe that
$$
T(d_{1})\cdots T(d_{j})T(g_{j,n}(d_{1},\dots,d_{j})) \ge T(n)
$$
from which we deduce that there is a choice of $j$ terms from the left side with a product larger than $T(n)^{\frac{j}{j+1}}$. Since $g_{j,n}$ is $k$-regular,
\begin{eqnarray*}
E_{j,k}(n) & \le & \frac{jk+1}{T(n)^{\frac{j}{j+1}}}\sum_{\substack{d_{i} \mid n \\ i=1,\dots,j \\ \gcd(d_{i_1},d_{i_2})=1 \\ (i_1 \neq i_2)}}T(d_{1})\cdots T(d_{j})\\
& = & (jk+1)\prod_{p^{v}\| n}\bigl((j+1)v^{\frac{j}{j+1}}\bigr)
\end{eqnarray*}
and the result follows from Lemma \ref{lem:3}.

\begin{rem}
Observe that
$$
(j+1)v^{\frac{j}{j+1}} < jv+1 \qquad(j \in \mathbb{N},\ v \ge 2).
$$
\end{rem}

\subsection{Proof of Theorem \ref{thm:3}a}

The argument revolves around the identity
$$
\mathsf{E}(\mathcal{D}_{n})=\sum_{e \mid n}\sum_{m \ge 1}U(e,m)^{2}
$$
where
$$
U(e,m):=|\{(d_{1},d_{2}) \in \mathcal{D}_{n}^2|\ d_{1}+d_{2}=me,\ \gcd(me,n)=e\}|.
$$
We first need an estimate for the number of solutions $(d_{1},d_{2},m) \in \mathcal{D}_{n}^{2}\times \mathbb{N}$ to the equation $d_{1}+d_{2}=me$ with $\gcd(me,n)=e$. Assume that $\gcd(d_{1},d_{2})=d$ which implies $d \mid e$. We then have
$$
d=\gcd(d_{1},d_{2})=\gcd(d_{1},me-d_{1})=\gcd(d_{1},me)=\gcd(d_{1},e)=\gcd(d_{2},e)
$$
and we deduce that $\frac{d_{1}}{d},\frac{d_{2}}{d} \mid \frac{n}{e}$ with $\gcd(\frac{d_{1}}{d},\frac{d_{2}}{d})=1$. It follows that we have at most $2^{\omega(e)}\kappa_{2}(n/e)$ such solutions, so that
\begin{equation}\label{4.1}
\sum_{m \ge 1}U(e,m) \le 3^{\omega(n)}\Bigl(\frac{2}{3}\Bigr)^{\omega(e)}.
\end{equation}

We also need an estimate for $U(e,m)$ for a fixed value of $m$. As above, we have $\gcd(d_{1},d_{2})=d$, $\frac{d_{1}}{d},\frac{d_{2}}{d} \mid \frac{n}{e}$ and $\gcd(\frac{d_{1}}{d},\frac{d_{2}}{d})=1$. We deduce that for a fixed $m$ we have $\frac{d_{1}}{d}\bigl(\frac{me}{d}-\frac{d_{1}}{d}\bigr) \mid \frac{n}{e}$ so that Théorème 2a of \cite{rdlb} implies that there are at most $2^{c\omega(n/e)}$ solutions $\frac{d_{1}}{d}$ with $c=\frac{\log 3}{\log 2}-\frac{2}{3}$. We then have at most $2^{c\omega(n)+(1-c)\omega(e)}$ choices for $d_{1}$ in total, so that
\begin{equation}\label{4.2}
U(e,m) \le 2^{c\omega(n)+(1-c)\omega(e)}.
\end{equation}

Let $\eta$ be a parameter that will be specified later. We write
\begin{eqnarray*}
\sum_{\substack{e \mid n \\ \omega(e) \ge (1-\eta)\omega(n)}}\sum_{m \ge 1}U(e,m)^{2} & \le & \sum_{\substack{e \mid n \\ \omega(e) \ge (1-\eta)\omega(n)}}2^{c\omega(n)+(1-c)\omega(e)}\sum_{m \ge 1}U(e,m)\\
& \le & \sum_{\substack{e \mid n \\ \omega(e) \ge (1-\eta)\omega(n)}}(3\cdot2^{c})^{\omega(n)}\Bigl(\frac{2^{2-c}}{3}\Bigr)^{\omega(e)}\\
\end{eqnarray*}
from \eqref{4.1} and \eqref{4.2}. For the small values of $\omega(e)$, we simply write
\begin{eqnarray*}
\sum_{\substack{e \mid n \\ \omega(e) < (1-\eta)\omega(n)}}\sum_{m \ge 1}U(e,m)^{2} & \le & \max_{\substack{e \mid n \\ \omega(e) < (1-\eta) \omega(n)}}2^{c\omega(n)+(1-c)\omega(e)}\sum_{e \mid n}\sum_{m \ge 1}U(e,m)\\
& = & \max_{\substack{e \mid n \\ \omega(e) < (1-\eta) \omega(n)}}2^{(2+c)\omega(n)+(1-c)\omega(e)}.
\end{eqnarray*}
We then choose $\eta:=0.2702949686\dots$ and the result follows.

\subsection{Proof of Theorem \ref{thm:3}b}

We consider the function
$$
r_{n}(m):=|\{(d_{1},d_{2}) \in \mathcal{D}_{n}^{2}:\ d_{1}+d_{2}=m\}|.
$$
It is enough to establish the estimate
$$
r_{n}(m) \ll \frac{\tau(n)}{\sqrt{\Omega_{2}(n)}}
$$
uniformly in $m$. We observe that $\frac{m}{2} \le \max(d_{1},d_{2}) \le m-1$ for each couple $(d_{1},d_{2})$ counted by $r_{n}(m)$ and that this maximum fixes at most 2 solutions, Therefore, this last inequality follows from Lemma \ref{lem:6}.

\subsection{Proof of Theorem \ref{thm:4}}

Given a congruence class $t$ among $1,\dots,q \pmod{q}$ and an integer $m$, we denote by $\lambda_{m}(t)$ the number of divisors of $m$ that are congruent to $t \pmod{q}$. The identity
\begin{equation}\label{eq:4.0}
\mathcal{H}(n,q) = \sum_{t=1}^{q}\lambda_{n}(t)^{2} \le \max_{t}\lambda_{n}(t)\sum_{t=1}^{q}\lambda_{n}(t) = \bigl(\max_{t}\lambda_{n}(t)\bigr)\tau(n),
\end{equation}
leads us to focus on an estimate for $\lambda_{n}:=\max_{t}\lambda_{n}(t)$. We note that for any factorization $n=ab$ in coprime integers $a$ and $b$ we have
$$
\lambda_{n}(t) = \sum_{d \mid b}\lambda_{a}(td^{-1}),
$$
from which we deduce
\begin{equation}\label{eq:4.00}
\lambda_{ab} \le \lambda_{a}\tau(b) \quad (\gcd(a,b)=1).
\end{equation}

Also, using Lemma \ref{lem:8}, we have
\begin{equation}\label{eq:4.4}
q \ge m^{1/4+\mu} \quad \Rightarrow \quad \lambda_{m} \ll \frac{1}{\mu^{3/2}}\qquad(0 < \mu \le 3/4)
\end{equation}
for every divisor $m$ of $n$.

\subsubsection*{Completion of the proof of Theorem \ref{thm:4}}

After this preparation, the proof is divided in three cases. The first case is when $q \ge n$, where we clearly have $\lambda_{n} = 1$ and therefore $\mathcal{H}(n,q)=\tau(n)$. The second case is when $n^{1/4} \le q < n$. Let $a$ be the divisor $n$ defined by $a=\min_{p^{\alpha} || n}\frac{n}{p^{\alpha}}$ and write $n=ab$. Given that $a \le n^{1-\frac{1}{\omega(n)}}$, it follows that $q \ge n^{1/4} \ge a^{\frac{1}{4}+\frac{1}{4\omega(n)}}$. From \eqref{eq:4.4} and the inequality $\tau(n) \ge 2^{\omega(n)}$, it follows that $\lambda_{a} \ll \omega(n)^{3/2} \ll (\log \tau(n))^{3/2}$. By \eqref{eq:4.00},
\begin{eqnarray*}
\lambda_{n} & \le & \lambda_{a}\tau(b)\\
& \ll & V(n)(\log \tau(n))^{3/2}
\end{eqnarray*}
and the result follows from inequality \eqref{eq:4.0}.

The third case is when $q < n^{1/4}$. We select a small $0 < \epsilon < \eta$ to be determined later. The first step is to define a factorization $n=ab$. Let us consider the pairs $(\rho_{i},\theta_{i})$, $i=1,\dots,\omega(n)$, defined by $(\alpha_{i}+1)=\tau(n)^{\rho_{i}}$ and $p_{i}^{\alpha_{i}}=n^{\theta_{i}}$. By Lemma \ref{lem:7}, there exists a permutation $\sigma$ such that
$$
\sum_{i=1}^{s}\rho_{\sigma_{i}} \le \sum_{i=1}^{s}\theta_{\sigma_{i}}
$$
for each $1 \le s \le \omega(n)$. We look at $\sigma$ as fixed and we choose the smallest $s$ such that
$$
\sum_{i=1}^{s}\theta_{\sigma_{i}} \ge 1-4\eta+\epsilon, \quad \mbox{so that} \quad \sum_{i=1}^{s-1}\theta_{\sigma_{i}} < 1-4\eta+\epsilon.
$$
We define
$$
b:=\prod_{i=1}^{s}p_{\sigma_{i}}^{\alpha_{\sigma_{i}}}.
$$
With the factorization $n=ab$ established, we note that
$$
\tau(b) = (\alpha_{\sigma_{s}}+1) \cdot \prod_{i=1}^{s-1}(\alpha_{\sigma_{i}}+1) \le 2V(n) \tau(n)^{1-4\eta+\epsilon}
$$
and also that, since $a=n/b$, we have $a \le n^{4\eta-\epsilon}$. Because $\eta > (4\eta-\epsilon)(1/4+\mu)$ holds with $\mu:=\epsilon/4$, we conclude that $\lambda_{a} \ll \frac{1}{\epsilon^{3/2}}$ by \eqref{eq:4.4}. By \eqref{eq:4.00},
\begin{eqnarray*}
\lambda_{n} & \le & \lambda_{a}\tau(b)\\
& \ll & \frac{1}{\epsilon^{3/2}}V(n) \tau(n)^{1-4\eta+\epsilon}
\end{eqnarray*}
in which we choose $\epsilon=\frac{1}{\log \tau(n)}$. We verify that it is either small enough, which is $\frac{1}{\log \tau(n)} \le \frac{\log q}{\log n}$, or the inequality is trivial. The result follows from \eqref{eq:4.0}.

\section{Proof of the corollaries}
\subsection{Proof of Corollary \ref{cor:1}}

We start by partitioning the solutions according to $d:=\gcd(d_{1},d_{2})$. By writing $d'_{i}=\frac{d_{i}}{d}$ for $i=1,2,3$, we obtain the equation $d'_{1}+d'_{2}=d'_{3}$ where each $d'_{i} \in \mathcal{D}_{n/d}$ and each divisor is coprime to the others. Consequently, we can apply Theorem \ref{thm:1}b to the integer $n/d$ with the function $g_{2,n/d}(d'_{1},d'_{2})=d'_{1}+d'_{2}$ which is 1-regular. We conclude that, for each fixed $d$, the number of solutions is at most $\kappa_{2}(n/d)^{1-\delta}$. Thus, the total number of solutions is at most
\begin{eqnarray*}
\sum_{d \mid n}\kappa_{2}(n/d)^{1-\delta} & = & \prod_{p^{v} \| n}\bigl(1+\kappa_{2}(p)^{1-\delta}+ \cdots+ \kappa_{2}(p^{v})^{1-\delta}\bigr)\\
& \le &\prod_{p^{v} \| n}(v+1)^{2-\delta} =\tau(n)^{2-\delta}.
\end{eqnarray*}
We invoke Lemma \ref{lem:5} to derive the second line.

\begin{rem}
We note that the function $g_{2,n/d}(d'_{1},d'_{2})$, as defined previously, is also strongly 2-regular. Thus, Theorem \ref{thm:2}a gives the estimate $2(1+2^{3/2})^{\omega(n)}$ when $n$ is squarefree. Since here $j=2$, Remark \ref{rem_imp} allows us to conclude that the estimate $(1+2^{3/2})^{\omega(n)}$ holds in this case.
\end{rem}

\subsection{Proof of Corollary \ref{cor:2}}

First, we present a preliminary remark. Let $p$ be a prime number for which $p^{V(n)} \| n$ and $g_{j,n}$ be a $k$-regular function satisfying $F_{j,k}(n)=E_{j,k}(n)$. For each tuple $(d_{1},\dots,d_{j})\in U_{g}$, we have either $p \nmid d_{1}\cdots d_{j}$ or that
$$
p \nmid \frac{d_{1}\cdots d_{j}g_{j,n}(d_{1},\dots,d_{j})}{d_{i}}
$$
for an $i \in \{1,\dots,j\}$. We deduce that
\begin{equation}\label{c2}
E_{j,k}(n) \le \frac{(jk+1)\kappa_{j}(n)}{jV(n)+1}.
\end{equation}

We are now ready for the argument. If $\omega(n) \ge \frac{\log \Omega(n)}{\delta_{j}\log(j+1)}$ is satisfied, then the result follows directly from Theorem \ref{thm:1}a. Furthermore, if there are at least $j+1$ distinct prime numbers $p$ such that $p^{v}\| n$ with $v \ge \Omega(n)^{1/(j+1)}$, then the result follows from Theorem \ref{thm:2}b. Thus, we can assume that neither of these properties holds. In this case, we have
\begin{eqnarray*}
V(n) \le \Omega(n) & \le & jV(n)+\omega(n)\Omega(n)^{1/(j+1)}\\
& \le & jV(n)+O(\Omega(n)^{1/(j+1)}\log\Omega(n)),
\end{eqnarray*}
from which it follows that $V(n) \asymp \Omega(n)$. The result then becomes a consequence of \eqref{c2}.

\subsection{Proof of Corollary \ref{cor:3}}

We define
$$
W_{n}(\theta) := \sum_{d \mid n}e(\theta d),
$$
where $e(x):=e^{2\pi i x}$. We can express
\begin{eqnarray*}
G(m) & = & \int_{0}^{1}W^{2}_{n}(\theta)\overline{W}_{n}(\theta)e(-m\theta)d\theta\\
& \le & \int_{0}^{1}\bigl|W^{2}_{n}(\theta)\overline{W}_{n}(\theta)\bigr|d\theta.
\end{eqnarray*}
Thus, the result follows from an application of the Cauchy-Schwarz inequality and Theorem \ref{thm:3}a.

\begin{rem}
We also have
$$
\mathcal{H}(n,q)=\frac{1}{q}\sum_{a=1}^{q}\bigl|W_{n}\Bigl(\frac{a}{q}\Bigr)\bigr|^{2}.
$$
\end{rem}

\section{Conclusion}

We can use Theorem \ref{thm:1}b to estimate the number of solutions to the equation $d_{1}d_{2}=d_{3}+1$ in integers $(d_{1},d_{2},d_{3}) \in \mathcal{D}_{n}^{3}$. It suffices to use and approach similar to the one used in Corollary \ref{cor:1}, partitioning the solutions according to $\gcd(d_{1},d_{2})$, which yields the same result. This example guided us in the early stages of this project.

A different type of example is provided by the functions $g_{2,n}(t_{i},t_{j})=t_{\frac{i+j}{2}}$ and $g_{2,n}(t_{i},t_{j})=t_{\lfloor\frac{i+j}{2}\rfloor}$, which are 1-regular and 2-regular, respectively. This is reminiscent of the function $g_{1,n}(t_{i})=t_{i+1}$, which is implicitly mentioned in the introduction. We can describe a general example by choosing $g_{j,n}$ to be an appropriate polynomial in $\mathbb{Z}[x_{1},\dots,x_{j}]$.

The results can be further extended by simultaneously considering several functions $g_{1,j,n},g_{2,j,n},\dots$ that satisfy a certain generalization of the concept of regularity.

\subsection*{Acknowledgement}

I would like to express my gratitude to Jean-Marie De Koninck for his valuable advice and encouragement. I also wish to thank Régis de la Bretèche for an engaging discussion related to this work.

{\sc D\'epartement de math\'ematiques et de statistique, Universit\'e Laval, Pavillon Alexandre-Vachon, 1045 Avenue de la M\'edecine, Qu\'ebec, QC G1V 0A6} \\
{\it E-mail address:} {\tt Patrick.Letendre.1@ulaval.ca}

\end{document}